\documentclass[11pt,a4paper]{article}
\usepackage[colorlinks=true,citecolor=blue,linkcolor=magenta]{hyperref}
\usepackage{amssymb}
\usepackage{amsmath}
\usepackage{amsthm}
\usepackage{amsfonts}
\usepackage{enumerate}
\usepackage{graphicx}
\usepackage{colortbl}
\usepackage{tikz}
\usepackage{ifthen}
\usepackage{float}
\usepackage{caption}
\usepackage{graphicx}
\usepackage{subcaption}
\usepackage{booktabs}
\usepackage[a4paper]{geometry}
\usepackage[linesnumbered, boxruled]{algorithm2e}
\usepackage{authblk}
\usepackage{relsize}

\usepackage[numbers,square,comma]{natbib} 

\newtheorem{theorem}{Theorem}
\newtheorem{definition}[theorem]{Definition}
\newtheorem{lemma}[theorem]{Lemma}
\newtheorem{corollary}[theorem]{Corollary}

\newtheorem{observation}[theorem]{Observation}

\SetKwInOut{Input}{Input}\SetKwInOut{Output}{Output}

\makeatletter
\renewcommand{\algocf@caption@boxruled}{%
  \hrule
  \hbox to \hsize{%
    \vrule\hskip-0.4pt
    \vbox{
       \vskip\interspacetitleboxruled%
       \unhbox\algocf@capbox\hfill
       \vskip\interspacetitleboxruled
       }%
     \hskip-0.4pt\vrule%
   }\nointerlineskip%
}%
\makeatother

\def\Nset{\mathbb{N}}

\newcommand{\plc}{P}
\newcommand{\placement}[1][]{(\xmin#1, \ymin#1, \xmax#1, \ymax#1)}
\newcommand{\singleRel}{\alpha}

\newcommand{\xmin}{x_{\min}}
\newcommand{\ymin}{y_{\min}}
\newcommand{\xmax}{x_{\max}}
\newcommand{\ymax}{y_{\max}}

\newcommand{\nat}{\mathbb{N}}
\newcommand{\real}{\mathbb{R}}

\newcommand{\bigO}{\mathcal{O}}

\newcommand{\stress}[1]{{\emph{#1}}}
\newcommand{\wrap}[1]{#1}

\newcommand{\smallClb}{6.828}
\newcommand{\bigCub}{11.091}

\definecolor{orange}{rgb}{1,0.9,0}
\definecolor{violet}{rgb}{0.8,0,1}
\definecolor{darkgreen}{rgb}{0,0.5,0}
\definecolor{grey}{rgb}{0.75,0.75,0.75}

\newcounter{arraylengthcounter}
\def\getNoOfElements#1#2{%
  \setcounter{arraylengthcounter}{0}%
  \foreach\arraylengthcounter in {#1}{\stepcounter{arraylengthcounter}}%
  \edef#2{\arabic{arraylengthcounter}}
}

\newcommand{\labelname}[1]{}

\newcommand{\drawPermutationImpl}[4]{
 \begin{pgfonlayer}{fg}    
 \getNoOfElements{#1}\Numx
 \getNoOfElements{#2}\Numy
 \foreach[count=\i] \val in {#1}
 {
  \draw[grey,densely dotted] (\i,0) -- (\i,{\Numy+1});
  \node[] at (\i,-1) {\strut \small $#3{\val}$};
 }
 \foreach[count=\i] \val in {#2}
 {
  \draw[grey,densely dotted] (0,\i) -- ({\Numx+1},\i);
  \node at (-1.5,\i) {\strut \small $#4{\val}$};
 }
 \foreach[count=\i] \val in {#1}
 {
  \foreach[count=\j] \valj in {#2}
  {
   \ifthenelse{\equal{\val}{\valj}}
   {
    \node[thick,circle,draw,scale=0.5] (\val) at (\i,\j) {}; 
   }
   {}
  }
 }
 \node (min) at (0,0) {\labelname{min}};
 \node (max) at ({\Numx+1},{\Numy+1}) {\labelname{max}};
 \end{pgfonlayer}
}

\colorlet{forbiddencol}{gray!20}
\colorlet{forbiddencoldark}{forbiddencol!94!black}
\newcommand{\forbiddencolor}{forbiddencol}

\newcommand{\drawPermutationTextRaw}[1]{#1}
\newcommand{\drawPermutationTextAddPi}[1]{\pi(#1)}
\newcommand{\drawPermutationTextAddRho}[1]{\rho(#1)}

\newcommand{\drawPermutation}[2]{
   \drawPermutationImpl{#1}{#2}{\drawPermutationTextRaw}{\drawPermutationTextAddPi}
}

\pgfdeclarelayer{bg}    
\pgfdeclarelayer{fg}    
\pgfsetlayers{bg,main,fg}  

\begin{document}

\author{Jannik Silvanus}
\author{Jens Vygen}
\affil{Research Institute for Discrete Mathematics, University of Bonn
}
\renewcommand\Authand{\qquad\qquad}
\title{Few Sequence Pairs Suffice: \\ { \smaller Representing All Rectangle Placements} }


\date{August 31, 2017}

\maketitle

\begin{abstract}
We consider representations of general non-overlapping placements of rectangles by spatial relations (west, south, east, north) of pairs of rectangles.
We call a set of representations \emph{complete} if it contains a representation of every placement of $n$ rectangles.

We prove a new upper bound of
$\mathcal{O}(\frac{n!}{n^6} \cdot (\frac{11+5 \sqrt 5}{2})^n)$
and a new lower bound of
$\Omega(\frac{n!}{n^4} \cdot (4 + 2 \sqrt2)^n)$
on the minimum cardinality of complete sets of representations.
A key concept in the proofs of these results are pattern-avoiding permutations.

The new upper bound directly improves upon the well-known sequence pair representation, which has size $(n!)^2$, by only considering a \emph{restricted} set of sequence pairs.
It implies theoretically faster algorithms for VLSI placement problems.
\end{abstract}

\section{Introduction}
\label{intro}

Axis-aligned non-overlapping placements of rectangles
can be characterized by the set of spatial relations (also called \stress{ABLR-relations} \cite{zhang2004theory} and \stress{HV-relations} \cite{murata1997mapping}) that hold for pairs of rectangles (west, south, east, north).
A \stress{representation} of such a placement consists of a satisfied spatial relation for every pair.
See Figure~\ref{intro::figure::placements}.

We call a set $R$ of representations \textit{complete} (for $n \in \nat$) if for every placement $\plc{}$ of any $n$ rectangles, $R$ contains a representation of $\plc{}$.
Naturally, one is interested in smallest complete sets of representations.
There is a trivial complete set of representations of size $4^{\binom{n}{2}}$.
The famous sequence-pair representation, first suggested by \cite{jerrum1985complementary} and later rediscovered by \cite{murata1996vlsi},
maps pairs of permutations to representations and achieves a size of ${(n!)}^2$.

Many other strategies have been proposed to obtain a complete set of representations, often only for placements with additional properties.
An overview is given in \cite{ChenChangFloorplanChapter}.
The previously best known upper bound was $\bigO(\frac{n!}{n^{4.5}}32^n)$, see Section~\ref{intro::related_work}.

In this paper, we establish the first nontrivial lower bound of $\Omega(\frac{n! c^n}{n^4} )$ and a new upper bound of $\bigO(\frac{n! C^n}{n^6} )$
on the minimum cardinality of complete sets of representations, where $c=4 + 2 \sqrt2 \geq \smallClb{}$ and $C=\frac{11+5 \sqrt 5}{2} \leq \bigCub{}$.
Both new bounds are proven by considering certain pattern-avoiding permutations.

\begin{figure}
\providecommand{\subfigwidth}{0.3\textwidth}

 \begin{subfigure}[b]{\subfigwidth}
  \centering
  \resizebox{.8\textwidth}{!}{%
  \begin{tikzpicture}
   \draw (1,1) [fill=gray!20]   rectangle (2.5,3) node[pos=.5] {$1$};
   \draw (2.5,1) [fill=gray!20] rectangle (4,2)   node[pos=.5] {$2$};
   \draw (1,3) [fill=gray!20]   rectangle (2.5,4) node[pos=.5] {$3$};
   \draw (2.5,2) [fill=gray!20] rectangle (4,4)   node[pos=.5] {$4$};
  \end{tikzpicture}
  }%
  \subcaption{A placement admitting two representations.}
 \label{intro::figure::placements::placement1}
 \end{subfigure}
\hfill
\begin{subfigure}[b]{\subfigwidth}
\centering
\begin{tabular}{lrr}
 Pair & $r$ & $r^{\prime}$ \\ \toprule
 $(1,2)$ & west  & west  \\
 $(1,3)$ & south & south \\
 $(1,4)$ & west  & west  \\
 $(2,3)$ & \stress{south} & \stress{east}  \\
 $(2,4)$ & south & south \\
 $(3,4)$ & west  & west  \\
\end{tabular}

\subcaption{Two representations of the placement on the left.}
 \label{intro::figure::placements::table}
\end{subfigure}
\hfill
 \begin{subfigure}[b]{\subfigwidth}
  \centering
  \resizebox{.8\textwidth}{!}{%
  \begin{tikzpicture}
   \draw (1,1) [fill=gray!20] rectangle (2,3) node[pos=.5] {$1$};
   \draw (2,1) [fill=gray!20] rectangle (4,2) node[pos=.5] {$2$};
   \draw (1,3) [fill=gray!20] rectangle (3,4) node[pos=.5] {$3$};
   \draw (3,2) [fill=gray!20] rectangle (4,4) node[pos=.5] {$4$};
  \end{tikzpicture}
  }%
  \subcaption{A placement represented by $r$, but not by $r^{\prime}$.}
 \label{intro::figure::placements::placement2}
 \end{subfigure}
 \caption{Representations of placements.}
 \label{intro::figure::placements}
\end{figure}
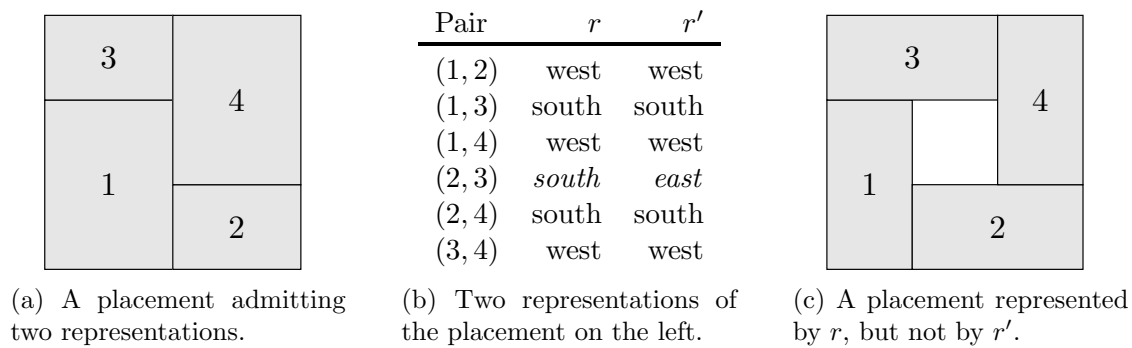


Given some rectangles, the set of placements satisfying a fixed representation forms a polyhedron with one inequality per rectangle pair.
Hence, it can be efficiently optimized over.
The theoretically fastest algorithms for VLSI placement problems such as half-perimeter wirelength optimization
 work by enumerating a (small) complete set of representations and computing, for each representation, an optimal represented placement.
Hence, constructing a smaller complete set of representations directly implies faster algorithms for these problems.
In practice, fast exact algorithms use a branch-and-bound method that directly enumerates the possible spatial relations (\cite{onodera1991branch}, \cite{funke2016exact}).
Moreover, many heuristics based on local search optimizing over a complete set of representations have been proposed (e.g., \cite{murata1996vlsi}, \cite{young2003twin}).
Both exact algorithms and heuristics could also benefit from a more compact complete set of representations.


\subsection{Related work}
\label{intro::related_work}

A closely related concept uses a \stress{floorplan} to represent the relative positions of rectangles.
A floorplan is a dissection of a rectangle by horizontal and vertical line segments into $m$ smaller rectangles, called \stress{rooms}, some of which may be marked as \stress{empty}.
Then, $n \leq m$ rectangles can be assigned bijectively to the nonempty rooms.
A floorplan without empty rooms is called \stress{mosaic} floorplan.
A mosaic floorplan that can be obtained by recursively splitting a room vertically or horizontally into two rooms is called \stress{slicing} floorplan.
The placement depicted in Figure~\ref{intro::figure::placements::placement1} corresponds to a slicing floorplan.
Contrary, the placement in Figure~\ref{intro::figure::placements::placement2} corresponds to a general, non-mosaic floorplan.
It can be turned into a non-slicing mosaic floorplan by filling the empty central room with a rectangle.

The structure of a floorplan can be captured by \emph{segment-room relations}: A segment $s$ and a room $r$ have the segment-room relation south if and only if $s$ contains the bottom edge of $r$, etc.
Then, we consider two floorplans as equivalent if there is a labeling of their rooms and segments which results in the same segment-room relations and which preserves empty rooms.
Note that some authors consider an assignment of the rectangles to the nonempty rooms to be part of a floorplan.
In \cite{murata1997mapping} (Property 5), it is shown that for each pair of rooms in a floorplan equivalence class,
one can deduce a spatial relation that is satisfied by each floorplan in this equivalence class.
This is proven by showing for each pair of rooms the existence of a sequence of segment-room relations that implies a spatial relation for the pair.
In the remainder of this section, when we speak of the number of certain floorplans, we mean the number of equivalence classes.

Using a bijection (\cite{ackerman2006bijection}) between Baxter permutations and mosaic floorplans,
their number is known to be $\Theta(\frac{8^n}{n^4})$ (first shown in \cite{yao2003floorplan}).
The same map, restricted to separable permutations, is a bijection to slicing floorplans, showing that the number of slicing floorplans is $\Theta(\frac{(3+\sqrt{8})^n}{n^{1.5}})$ (also first shown in \cite{yao2003floorplan}).

General floorplans may contain an arbitrary number of empty rooms.
Young et al.~\cite{young2003twin} call an empty room \stress{reducible} if it can be merged with adjacent rooms while keeping the spatial relations of the remaining nonempty rooms implied by the floorplan.
For example, the empty room in the floorplan corresponding to the placement in Figure~\ref{intro::figure::placements::placement2} is not reducible.
On the contrary, all rooms in the floorplan corresponding to the placement in
Figure~\ref{intro::figure::placements::placement1} would be reducible if empty.
The best upper bound of $\bigO(\frac{32^n}{n^{4.5}})$ on
the number of general floorplans with $n$ rectangles (occupied rooms)
and without reducible empty rooms was shown in \cite{shen2003bounds},
and no stronger lower bound than the number of mosaic floorplans is known.

Property 1 and Theorem 3 in \cite{murata1997mapping} imply that for each placement of $n$ rectangles,
there exists a floorplan equivalence class with $n$ nonempty rooms and an assignment of the rectangles into the nonempty rooms
such that each pair of rectangles satisfies the spatial relation implied by their rooms in the floorplan equivalence class.
Hence, an upper bound $U(n)$ on the number of general floorplans with $n$ occupied rooms and without reducible empty rooms implies an upper bound of $U(n) \cdot n!$ on the minimum size of a complete set of representations for $n$ rectangles.
However, it is unknown whether lower bounds can be transferred in the same way.

\subsection{Definitions}

\newcommand{\west}{\text{west}}
\newcommand{\east}{\text{east}}
\newcommand{\south}{\text{south}}
\newcommand{\north}{\text{north}}

Let $n\in\Nset$.
We denote by $[n]$ the set of integers $\{1, \ldots, n\}$.

A \stress{placement} is a tuple of functions $\plc{} = \placement$ from $[n]$ to $\mathbb{R}$ with, for $i \in [n]$,
\begin{enumerate}[~~~~~(i)]
 \item $\xmin(i) < \xmax(i)$, and
 \item $\ymin(i) < \ymax(i)$.
\end{enumerate}

We often call the elements of $[n]$ \stress{rectangles},
and refer to $n$ as the \stress{size} of $\plc{}$.
A placement is called \stress{feasible} if for all $1\le i<j\le n$ at least one of the following holds:
\begin{center}
\begin{tabular}{r@{}c@{}lll}
$\xmax(i)$ & $\ \le \ $ & $\xmin(j)$  && ($i$ is \stress{west} of $j$) \\
$\ymax(i)$ & $\ \le \ $ & $\ymin(j)$  && ($i$ is \stress{south} of $j$) \\
$\xmax(j)$ & $\ \le \ $ & $\xmin(i)$  && ($i$ is \stress{east} of $j$) \\
$\ymax(j)$ & $\ \le \ $ & $\ymin(i)$  && ($i$ is \stress{north} of $j$)
\end{tabular}
\end{center}

A function $r: [n]^2\setminus\{(i,i):i\in[n]\} \to \{\west,\south,\east,\north\}$
is a \stress{representation} of (or \stress{represents}) a feasible placement
if the following statements hold for all $i,j\in [n]^2$ with $i\not=j$:
\begin{center}
\begin{tabular}{r@{}c@{}lcll}
$r(i,j)$&$\ = \ $&$\west$ & $\Rightarrow$ & $i$ is west of $j$ \\
$r(i,j)$&$\ = \ $&$\south$ & $\Rightarrow$ & $i$ is south of $j$ \\
$r(i,j)$&$\ = \ $&$\east$ & $\Rightarrow$ & $i$ is east of $j$ \\
$r(i,j)$&$\ = \ $&$\north$ & $\Rightarrow$ & $i$ is north of $j$
\end{tabular}
\end{center}

A \stress{complete set of representations} for $n$ is a set $R$ such that for all feasible placements $\plc{}$
of size $n$ there exists an element of $R$ that represents $\plc{}$.

How small can a complete set of representations be?
Obviously it needs to have cardinality at least $n!$ because for placements in which all
rectangles have identical $y$-coordinates, we must represent all $n!$ horizontal orders.
A trivial upper bound is $4^{\binom{n}{2}}$ because for each unordered pair there are four possibilities.
In this paper we prove a new lower bound and a new upper bound.

\section{Preliminaries on permutations}

Both for the lower and the upper bound, we will restrict to pairs of permutations with certain properties.
With any permutation $\pi$ on $[n]$ we associate a strict total order $<_{\pi}$ by defining $i  <_{\pi} j \iff \wrap{\pi}(i) < \wrap{\pi}(j)$ for $i,j\in [n]$.

\subsection{Plane permutations}

The definitions below follow \cite{bousquet2007forest}.
\begin{definition}
 We call a permutation $\pi$ on $[n]$ \stress{plane} if it avoids the pattern $21\bar 3 54$,
 i.e., if for all indices $i < j < l < m \in [n]$ with $j <_{\pi} i <_{\pi} m <_{\pi} l$, there exists
 an index $k$ with $j < k < l$ with $i <_{\pi} k <_{\pi} m$.
\end{definition}

Figure~\ref{definition:plane_perm:figure} gives an illustration of the forbidden pattern for plane permutations.

\begin{figure}
\centering
\providecommand{\subfigwidth}{0.47\textwidth}
 \begin{subfigure}[t]{\subfigwidth}
  \centering
  \resizebox{0.95\textwidth}{!}{%
   \begin{tikzpicture}[scale=0.8]
      \drawPermutation{i,j,l,m}{j,i,m,l}
       \fill [forbiddencol] (j |- i) rectangle (l |- m);
   \end{tikzpicture}
   }
   \subcaption{The forbidden configuration for plane permutations: If there are elements as shown and no elements in the gray area, then $\pi$ is not plane.}
   \label{definition:plane_perm:figure}
 \end{subfigure}
\hfill
 \begin{subfigure}[t]{\subfigwidth}
  \centering
  \resizebox{0.95\textwidth}{!}{%
   \begin{tikzpicture}[scale=.55]
      \drawPermutation{1,2,3,4,5,6,7,8}{5,1,6,8,2,4,3,7}
      \draw [thick,->] (1) -- (2);
      \draw [thick,->] (1) -- (6);
      \draw [thick,->] (2) -- (3);
      \draw [thick,->] (2) -- (4);
      \draw [thick,->] (3) -- (7);
      \draw [thick,->] (4) -- (7);
      \draw [thick,->] (5) -- (6);
      \draw [thick,->] (6) -- (7);
      \draw [thick,->] (6) -- (8);
   \end{tikzpicture}
   }
   \subcaption{The natural embedding of $G_{\pi}$ for the plane permutation $\pi = (2,5,7,6,1,3,8,4)$.
   }
   \label{definition:natural_embedding:figure}
 \end{subfigure}
 \caption{Illustrations on plane permutations. The elements are ordered on the x- and y-axis according to their relative order in $<$ and $<_\pi$, respectively.}
\end{figure}
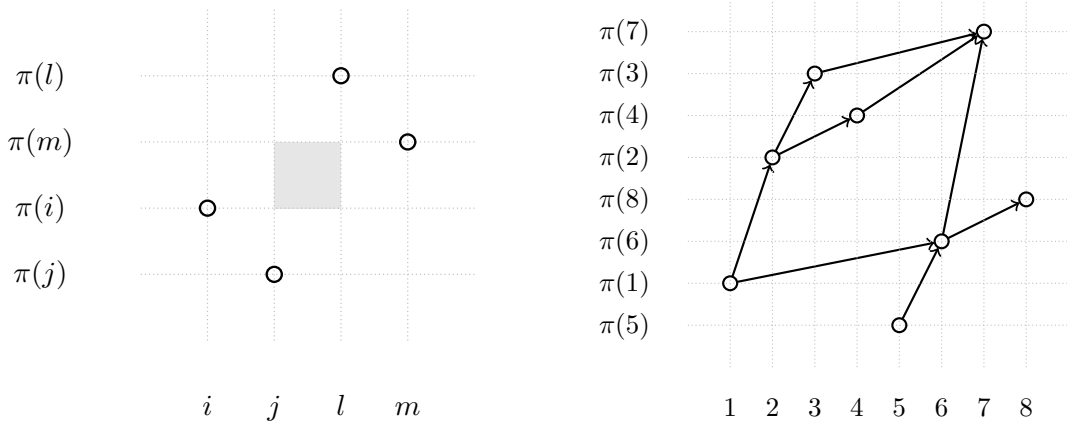

\begin{definition}
\label{perm:graph}
 Given a permutation $\pi$ on $[n]$, we define an acyclic directed graph \stress{$G_{\pi}$} with vertex set $[n]$, whose
 arc set $A(G_{\pi})$ consists of exactly the pairs $(i, j)$ with
 \begin{enumerate}[(i)]
  \item $i < j \text{ and } i <_{\pi} j$, and \label{perm:graph:first}
  \item there is no $k$ with $i < k < j$ and $i <_{\pi} k <_{\pi} j$. \label{perm:graph:second}
 \end{enumerate}
\end{definition}

\begin{observation} \label{perm:graph:reachability:observation}
  Let $\pi$ be a permutation on $[n]$ and $1 \leq i,j \leq n$.
  Then $j$ is reachable from $i$ in $G_\pi$ iff $i \le j$ and $\pi(i) \le \pi(j)$.
\end{observation}

To explain the name ``plane'', one can define a \stress{natural embedding} of $G_{\pi}$ by drawing $i$ in $(i, \pi(i)) \in \real^2$ and
 drawing all arcs as straight lines (cf.\ Figure~\ref{definition:natural_embedding:figure}).
One can show that $\pi$ is plane iff the natural embedding of $G_{\pi}$ is planar, but we will not need this.

We will use the following recent result.

\begin{theorem}[\cite{2017arXiv170204529B}] \label{thm_numberofplaneperm}
The number of plane permutations on $[n]$ is $\Theta \bigl( \frac{C^n}{n^6} \bigr)$, where $C = \frac{11+5 \sqrt 5}{2}$.
\end{theorem}

\subsection{Biplane permutations}

\begin{definition}
 Let $\pi$ be a permutation on $[n]$. The permutation $-\pi$ is defined by
 \begin{align*}
  -\pi(i) := n + 1 - \pi(i)
 \end{align*}
 for $i \in [n]$.
\end{definition}

\begin{observation}
\label{lowerbound:minuspi:reachability:observation}
 Let $\pi$ be a permutation on $[n]$, and let $1 \leq i < j \leq n$.
 Then $j$ is reachable from $i$ in $G_{\pi}$ iff $j$ is not reachable from $i$ in $G_{-\pi}$.
\end{observation}

For example, for the permutation $\pi$ in Figure \ref{definition:natural_embedding:figure} one can see that $-\pi$ is not plane.

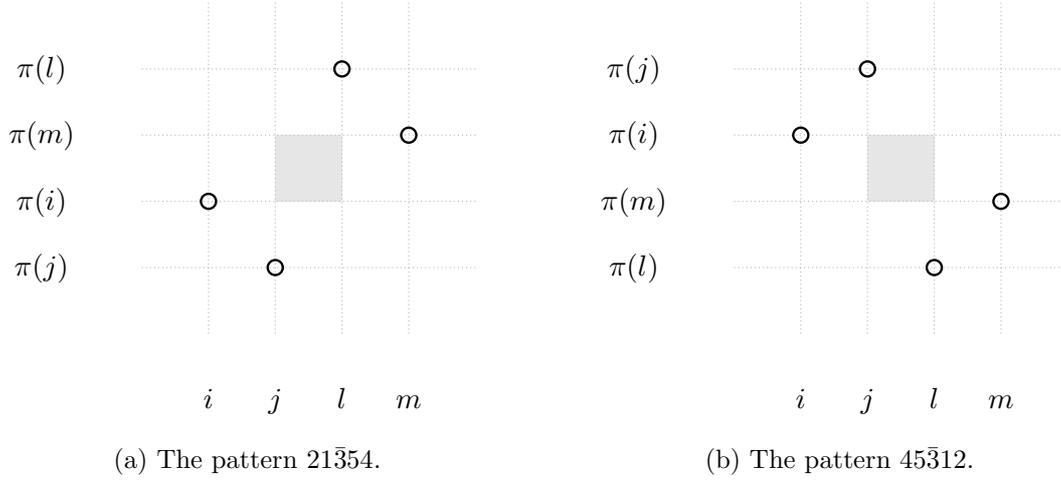
\begin{figure}[ht]
\centering
\providecommand{\subfigwidth}{0.47\textwidth}
\begin{subfigure}[b]{\subfigwidth}
  \centering
  \resizebox{0.95\textwidth}{!}{%
   \begin{tikzpicture}[scale=0.8]
      \drawPermutation{i,j,l,m}{j,i,m,l}
      \fill [\forbiddencolor] (j |- i) rectangle (l |- m);
   \end{tikzpicture}
   }
   \caption{The pattern $21\bar{3}54$.}
\end{subfigure}
\hfill 
\begin{subfigure}[b]{\subfigwidth}
  \centering
  \resizebox{0.95\textwidth}{!}{%
   \begin{tikzpicture}[scale=0.8]
      \drawPermutation{i,j,l,m}{l,m,i,j}
       \fill [\forbiddencolor] (j |- i) rectangle (l |- m);
   \end{tikzpicture}
   }
   \caption{The pattern $45\bar{3}12$.}
\end{subfigure}
\caption{Forbidden patterns of biplane permutations. Gray areas are assumed to be empty.}
\label{lowerbound:biplane_perm:figure}
\end{figure}

\begin{definition}
 Let $\pi$ be a permutation on $[n]$. We call $\pi$ \stress{biplane} if $\pi$ avoids the patterns $21\bar{3}54$ and $45\bar{3}12$.
\end{definition}

The patterns forbidden in biplane permutations are illustrated in Figure~\ref{lowerbound:biplane_perm:figure}.

\begin{observation}
 Let $\pi$ be a permutation. Then, the following statements are equivalent:
 \begin{enumerate}[(i)]
  \item $\pi$ is biplane
  \item $-\pi$ is biplane
  \item $\pi$ and $-\pi$ are plane
 \end{enumerate}
\end{observation}

The number of biplane permutations is also well-understood:

\begin{theorem}[\cite{DBLP:journals/combinatorics/AsinowskiBBMP13}] \label{thm_numberofbiplaneperm}
The number of biplane permutations on $[n]$ is $\Theta(\frac{c^n}{n^4})$, where $c = 4 + 2 \sqrt2$.
\end{theorem}

\section{Upper bound}

In this section we show a better upper bound on the size of complete sets of representations.
We first review the sequence pair representation of Jerrum \cite{jerrum1985complementary}
(rediscovered by \cite{murata1996vlsi}), because our proof will build on it.


\subsection{Sequence pairs}

Given a pair $(\pi,\rho)$ of permutations on $[n]$ (called a \stress{sequence pair}),
we define $r_{\pi,\rho}$ by
\begin{center}
\begin{tabular}{r@{}c@{}lcll}
$r_{\pi,\rho}(i,j)$&$\ = \ $&$\west$ & if & $i  <_{\pi} j$ and $j <_{\rho} i$ \\
$r_{\pi,\rho}(i,j)$&$\ = \ $&$\south$ & if & $i  <_{\pi} j$ and $i <_{\rho} j$ \\
$r_{\pi,\rho}(i,j)$&$\ = \ $&$\east$ & if & $j  <_{\pi} i$ and $i <_{\rho} j$ \\
$r_{\pi,\rho}(i,j)$&$\ = \ $&$\north$ & if & $j  <_{\pi} i$ and $j <_{\rho} i$
\end{tabular}
\end{center}

\begin{theorem}[\cite{jerrum1985complementary}] \label{thm_jerrum}
The set of functions $r_{\pi,\rho}$ for all sequence pairs $(\pi,\rho)$ is a complete set of representations.
\end{theorem}


Let us first give a new short proof of this famous result because this will be the basis for our improved upper bound.

Let $P=\placement$ be a feasible placement; this will be fixed for most of this section.
Given $P$, we first define four strict partial order
relations $N,S,E,W\subseteq [n]^2$ (north, south, east, west) by
\begin{eqnarray*}
N &:=& \left\{(a,b)\in[n]^2 \mid \ymax(b) \le \ymin(a) \right\}, \\
S &:=& \left\{(a,b)\in[n]^2 \mid \ymax(a) \le \ymin(b) \right\}, \\
E &:=& \left\{(a,b)\in[n]^2 \mid \xmax(b) \le \xmin(a) \right\}, \\
W &:=& \left\{(a,b)\in[n]^2 \mid \xmax(a) \le \xmin(b) \right\}.
\end{eqnarray*}
We have $N\cap S=\emptyset$,
$E\cap W=\emptyset$,
$(a,b)\in N\cup S\cup E\cup W$ for all $a\not=b$,
$(a,b)\in N$ iff $(b,a)\in S$, and
$(a,b)\in E$ iff $(b,a)\in W$.

Define two digraphs $G_1$ and $G_2$, both with vertex set $[n]$, and with arc sets:
\begin{eqnarray*}
A(G_1) &:=& (S\setminus E)\cup (W\setminus N) \\
A(G_2) &:=& (S\setminus W)\cup (E\setminus N)
\end{eqnarray*}

\begin{lemma}
\label{g1acyclic}
$G_1$ and $G_2$ are acyclic.
\end{lemma}

\begin{proof}
By symmetry, it suffices to consider $G_1$ (for $G_2$ exchange $W$ and $E$).

Let $a,b,c\in[n]$ such that $(a,b)\in S\setminus E$ and $(b,c)\in W\setminus N$.
As $E$ is transitive, $(a,b)\notin E$ and $(c,b)\in E$ imply $(a,c)\notin E$.
Similarly, $(b,a)\in N$ and $(b,c)\notin N$ imply $(a,c)\notin N$.
We conclude that $(a,c)\notin E\cup N$, implying
$(a,c)\in (S\setminus E) \cup (W\setminus N)=A(G_1)$.

Analogously,
if $(a,b)\in W\setminus N$ and $(b,c)\in S\setminus E$,
then $(a,c)\in (S\setminus E)\cup (W\setminus N)=A(G_1)$.

Therefore, for any $a,b\in[n]$ and any
shortest path from $a$ to $b$ in $G_1$,
either all arcs are in $S\setminus E$ or all arcs are in $W\setminus N$.
Hence we have $(a,b)\in S\cup W$ whenever $b$ is reachable from $a$.
This implies $(b,a)\in N\cup E$,
and therefore $(b,a)$ is not an arc of $G_1$.
So $G_1$ is indeed acyclic.
\end{proof}

\begin{lemma}
\label{sequencepairconsistent}
Let
$\pi$ and $\rho$ be topological orders of $G_1$ and $G_2$, respectively.
Then $r_{\pi,\rho}$ represents $\plc{}$.
\end{lemma}

\begin{proof}
If $a <_{\pi} b$ and $a <_{\rho} b$,
then $(b,a)$ is neither an arc of $G_1$ nor of $G_2$.
Hence $(b,a) \notin A(G_1) = (S\setminus E)\cup (W\setminus N)$, so $(b,a) \in N \cup E$.
Similarly, as $(b,a) \notin A(G_2) = (S\setminus W)\cup (E\setminus N)$, we have $(b,a) \in N \cup W$.
This means $(b,a) \in (N \cup E) \cap (N \cup W) = N \cup (W \cap E) = N$, so $(a,b) \in S$.

If $a <_{\pi} b$ and $b <_{\rho} a$,
then $(b,a)$ is not an arc of $G_1$ and $(a,b)$ is not an arc of $G_2$.
Again, by $(b,a) \notin A(G_1)$, it follows that $(b,a) \in N \cup E$, and hence $(a,b) \in S \cup W$.
Moreover, as $(a,b) \notin A(G_2)$, we have $(a,b) \in N \cup W$.
Hence $(a,b) \in (S \cup W) \cap (N \cup W) = (S \cap N) \cup W = W$.
\end{proof}

 This proves Theorem \ref{thm_jerrum} and hence the well-known $(n!)^2$ upper bound.

 \subsection{Augmented digraphs}

Now we improve on it by adding some arcs to the digraphs $G_1$ and $G_2$:
 \begin{eqnarray*}
A(G_1') &:=& A(G_1)\cup \{ (a,b)\in N\cap W : a \text{ not reachable from } b \text{ in } G_1  \} \\
A(G_2') &:=& A(G_2)\cup \{ (a,b)\in N\cap E : a \text{ not reachable from } b \text{ in } G_2  \}
\end{eqnarray*}

\begin{lemma}
\label{g1primeacyclic}
$G_1'$ and $G_2'$ are acyclic.
\end{lemma}

\begin{proof}
Again it suffices to consider $G_1'$ (for $G_2'$ again exchange $W$ and $E$).

Suppose $G_1'$ contains a circuit. Consider a circuit $C$ with smallest number of arcs.
Of course, $C$ must contain at least two arcs from $A(G_1')\setminus A(G_1)$ because $G_1$ is acyclic (Lemma \ref{g1acyclic})
and any single added arc does not create a circuit by construction.

Let $v_k,v_{k-1},\ldots,v_1,v_0,b$ be the vertices of a path in $C$
in which only the last arc $(v_0,b)$ does not belong to $G_1$.

\noindent
{\bf Claim:} $(v_i,b)\in N\cap W$ for all $i=0,\ldots,k$.

We show the Claim by induction on $i$.
It is true for $i=0$ because $(v_0,b)\in A(G_1')\setminus A(G_1)$. Let now $i\ge 1$.
As $(v_i,v_{i-1})$ is an arc of $G_1$, $(v_i,v_{i-1})\notin E$.
Moreover, $(b,v_{i-1})\in E$ by the induction hypothesis.
As $E$ is transitive, $(v_i,b)\notin E$.

Now $(v_i,b)$ is not an arc of $G_1$ because $C$ is a shortest circuit.
As $(v_i,b)\notin E$ , this implies $(v_i,b)\in N$.

Finally suppose that $(v_i,b)\notin W$.
Then $(b,v_i)\in S\setminus E$, and hence $(b,v_i)\in A(G_1)$.
Then $b,v_i,v_{i-1},\ldots,v_0$ is a path from $b$ to $v_0$ in $G_1$.
This is a contradiction to the fact that $(v_0,b)\in A(G_1')\setminus A(G_1)$.
The Claim is proved.

Now let $(a_i,b_i)$, $i=1,\ldots,l$, be the arcs of $C$ that do not belong to $G_1$.
We had $l\ge 2$, and by the Claim $(b_{i-1},b_i)\in N\cap W$ for all $i=1,\ldots,l$ (where $b_0:=b_l$).
This is impossible because $N$ and $W$ are strict partial orders.
\end{proof}

We now consider topological orders of $G_1'$ and $G_2'$. Since we only added arcs,
Lemma \ref{sequencepairconsistent} still applies. We will show that
only certain (much fewer) sequence pairs can occur as topological orders of $G_1'$ and $G_2'$.

\subsection{Bad quartets}

For a sequence pair $(\pi,\rho)$ we say that $(a,b,c,d)\in[n]^4$ is a \stress{bad quartet}
if the following three conditions hold:
\begin{itemize}
\item $a <_{\pi} b <_{\pi} c <_{\pi} d$,
\item $b <_{\rho} a <_{\rho} d <_{\rho} c$,
\item there is no $e\in[n]$ with  $b <_{\pi} e <_{\pi} c$ and $a <_{\rho} e <_{\rho} d$.
\end{itemize}

\begin{observation} \label{badquartetplane}
The number of sequence pairs without a bad quartet is $n!$ times the number of plane permutations.
\end{observation}

We call a bad quartet $(a,b,c,d)$ \stress{extreme} if
there is no $e\in[n]$ with $b <_{\pi} e <_{\pi} c$ and
there is no $f\in[n]$ with $a <_{\rho} f <_{\rho} d$.
In other words: $b$ and $c$ are consecutive in $\pi$, and $a$ and $d$ are consecutive in $\rho$.
See Figure \ref{figbadquartet}.

   \begin{figure}[ht]
   \centering
   \resizebox{0.46\textwidth}{!}{%
   \begin{tikzpicture}[scale=0.8]
      \drawPermutationImpl{a,b,c,d}{b,a,d,c}{\drawPermutationTextAddPi}{\drawPermutationTextAddRho}
      \fill [forbiddencoldark] (min |- a) rectangle (max |- d);
      \fill [forbiddencoldark] (b |- min) rectangle (c |- max);
      \fill [forbiddencol] (b |- a) rectangle (c |- d);
   \end{tikzpicture}
   }
\caption{A bad quartet $(a,b,c,d)$.
The light gray square in the center is empty by the third condition in the definition of bad quartets.
The bad quartet is extreme iff the four darker areas are empty, too.
\label{figbadquartet}}
 \end{figure}
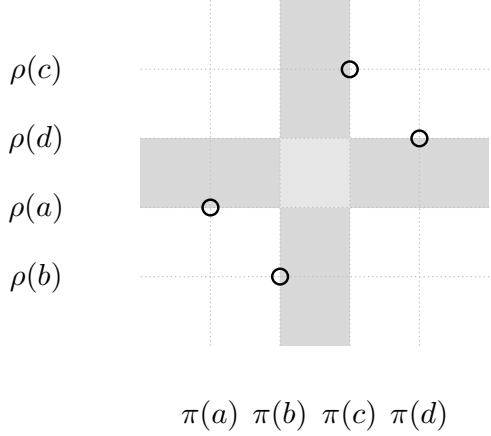

\begin{lemma}
\label{extremebadquartets}
If a sequence pair $(\pi,\rho)$ has a bad quartet, it has an extreme bad quartet.
\end{lemma}

\begin{proof}
Let $d_{\pi}(a,b):=|\{e\in[n]:  a <_{\pi} e <_{\pi} b \}|$
denote the number of elements in between $a$ and $b$ in permutation $\pi$.
Let $(a,b,c,d)$ be a bad quartet such that $\Phi:=d_{\pi}(b,c)+d_{\rho}(a,d)$ is minimum.
If $\Phi$ is zero, $(a,b,c,d)$ is extreme.

Otherwise, we consider two cases.
Suppose first that there is an $e\in[n]$ with $b <_{\pi} e <_{\pi} c$.
If $e<_{\rho} a$, then $(a,e,c,d)$ is a bad quartet with smaller $\Phi$.
Otherwise, since $(a,b,c,d)$ is a bad quartet, we have $d<_{\rho} e$, and $(a,b,e,d)$ is a bad quartet with smaller $\Phi$.

Secondly, suppose that there is an $e\in[n]$ with $a <_{\rho} e <_{\rho} d$.
If $c<_{\pi} e$, then $(a,b,c,e)$ is a bad quartet with smaller $\Phi$.
Otherwise, we have $e <_{\pi} b$, and $(e,b,c,d)$ is a bad quartet with smaller $\Phi$.
\end{proof}

\subsection{New upper bound}

Now we get the better upper bound:

\begin{lemma}
\label{nobadquartet}
Let $\pi$ and $\rho$ be topological orders of $G_1'$ and $G_2'$, respectively.
Then $(\pi,\rho)$ has no bad quartet.
\end{lemma}

\begin{proof}
First note that $r_{\pi,\rho}$ represents $P$ by Lemma \ref{sequencepairconsistent}.
If there is a bad quartet, there is an extreme one by Lemma \ref{extremebadquartets}.
So suppose that $(a,b,c,d)$ is an extreme bad quartet:
Then, $a <_{\pi} b <_{\pi} c <_{\pi} d$ and there is no $e\in[n]$ with $b <_{\pi} e <_{\pi} c$.
Moreover, $b <_{\rho} a <_{\rho} d <_{\rho} c$ and there is no $f\in[n]$ with $a <_{\rho} f <_{\rho} d$.

As $r_{\pi,\rho}$ represents $P$, we have $(a,c),(a,d),(b,c),(b,d)\in S$ and $(a,b),(c,d)\in W$.

\noindent
{\bf Claim 1:} $(c,b)\notin W$.

Suppose that $(c,b)\in W$. Then $(b,c)\in E$, and thus $(b,c)\notin A(G_1)$.
Therefore $c$ is not reachable from $b$ in $G_1$ as
any vertex on a $b$-$c$-path would have to be in between $b$ and $c$ in the topological order $\pi$.
But then $(c,b)\in N\cap W$ would be an arc of $G_1'$, contradicting $b <_{\pi} c$.

\noindent
{\bf Claim 2:} $(a,d)\notin W$.

Suppose that $(a,d)\in W$. Then $(a,d)\notin A(G_2)$.
Therefore $d$ is not reachable from $a$ in $G_2$ as
any vertex on an $a$-$d$-path would have to be in between $a$ and $d$ in the topological order $\rho$.
But then $(d,a)\in N\cap E$ would be an arc of $G_2'$, contradicting $a <_{\rho} d$.

The two claims are proved. However, they contradict each other:
together with $(a,b),(c,d)\in W$ they imply
$\xmax(a) \le \xmin(b) < \xmax(c) \le \xmin(d) < \xmax(a)$.
\end{proof}

We conclude:

\begin{theorem}
Let $C = \frac{11+5 \sqrt 5}{2} \leq \bigCub{}$. There is a complete set of representations for $n$ with $\bigO(n! \cdot \frac{C^n}{n^6})$ elements.
\end{theorem}

\begin{proof}
By Lemmata \ref{sequencepairconsistent}, \ref{g1primeacyclic}, and \ref{nobadquartet},
the set of functions $r_{\pi,\rho}$ for all sequence pairs $(\pi,\rho)$ without bad quartets is complete.
By Observation \ref{badquartetplane}, the number of sequence pairs without bad quartets is $n!$ times the number of plane permutations.
By Theorem \ref{thm_numberofplaneperm}, the number of plane permutations is $\Theta(\frac{C^n}{n^6})$.
\end{proof}

Note that this result not only implies an improved asymptotic behavior compared to classical sequence pairs,
but also yields a strict improvement for all $n \geq 4$.

\section{Lower bound} \label{lowerbound:section}

\newcommand{\downrel}{\downarrow}
\newcommand{\uprel}{\uparrow}
\newcommand{\floor}{F}
\newcommand{\ceil}{C}
\newcommand{\project}{\operatorname{pr}}
\newcommand{\projectset}{\operatorname{PR}}
\newcommand{\projectsetup}{\projectset{}_{\uparrow}}
\newcommand{\projectsetdown}{\projectset{}_{\downarrow}}
\newcommand{\projectup}{\project{}_{\uparrow}}
\newcommand{\projectdown}{\project{}_{\downarrow}}
\newcommand{\vertices}{V}
\newcommand{\downvertices}{V_{\downarrow}}
\newcommand{\upvertices}{V_{\uparrow}}
\newcommand{\downgraph}{G_{\downarrow}}
\newcommand{\bbox}{BB}
\newcommand{\open}[1]{\operatorname{op}(#1)}
\newcommand{\close}[1]{\operatorname{cl}(#1)}
\newcommand{\stepline}[1]{S_{#1}}
\newcommand{\steppoint}[1]{s_{#1}}
\newcommand{\pparensize}{\big}


In this section, we prove that every complete set of representations
has $\Omega \bigl( n! \cdot \frac{(4 + 2 \sqrt2)^n}{n^4} \bigr)$ elements.
The idea is to generate a large number of feasible placements which all allow only one ``useful'' representation,
where no representation occurs twice.
We construct these placements using biplane permutations, which have been examined in
\cite{DBLP:journals/combinatorics/AsinowskiBBMP13}.
The authors of~\cite{DBLP:journals/combinatorics/AsinowskiBBMP13} study orders on \stress{segments} of floorplans,
which have a very similar structure to the \stress{rectangles} in the placements considered in this section.

\subsection{Forcing placements and canonical representations}

\begin{definition}
 Let $\plc{}$ be a feasible placement of size $n$ and $i,j\in[n]$ with $i \neq j$.
 We say that a spatial relation $\singleRel{} \in \{ \west,\south,\east,\north \}$  is \stress{forced} for $(i, j)$ if there is a sequence
 of rectangles $i = a_1, \ldots, a_k = j$ such that for all $1 \leq m < k$, the only spatial relation of $(a_m, a_{m+1})$ in $\plc{}$ is $\singleRel{}$.
\end{definition}

\begin{observation}
 Let $\plc{}$ be a feasible placement of size $n$ and $i,j \in [n]$ with $i\not=j$.
 If only one spatial relation holds for $(i,j)$ in $\plc{}$, then this spatial relation is forced for $(i,j)$.
\end{observation}

\begin{observation}
\label{forced::transitivity::observation}
 Let $\plc{}$ be a feasible placement of size $n$, let $i,j,k \in [n]$ and let $\singleRel{} \in \{ \west,\south,\east,\north \}$ be a spatial relation.
 If $\singleRel{}$ is forced for $(i,j)$ and $(j,k)$ in $\plc{}$, then $\singleRel{}$ is also forced for $(i,k)$ in $\plc{}$.
\end{observation}

\begin{lemma}
\label{lowerbound:forced_relation:at_most_one:lemma}
 Let $\plc{} = \placement$ be a feasible placement of size $n$ and $i,j \in [n]$ with $i\neq j$.
 Then there is at most one forced spatial relation for $(i,j)$ in $\plc{}$.
\end{lemma}

\begin{proof}
Let $(\pi, \rho)$ be a sequence pair
such that $r_{\pi,\rho}$ represents $\plc{}$ (cf.\ Theorem \ref{thm_jerrum}).

First assume that $i$ west of $j$ is forced.
Then, there is a sequence $i = a_1, \ldots, a_k = j$ such that for all $1 \leq m < k$, the only spatial relation of $(a_m, a_{m+1})$ in $\plc{}$ is west.
Since $r_{\pi, \rho}$ represents $\plc{}$, for all $1 \leq m < k$, we have
$a_m <_{\pi} a_{m+1}$ and $a_{m+1} <_{\rho} a_{m}$.
Hence, we have $i <_{\pi} j$ and $j <_{\rho} i$.

Using the same argument, if $i$ south of $j$ is forced, we have $i <_{\pi} j$ and $i <_{\rho} j$, etc.
This shows that at most one spatial relation can be forced for $(i,j)$.
\end{proof}

\begin{definition}
 Let $\plc{}$ be a placement.
 We call $\plc{}$ \stress{forcing} if for all rectangles $i \neq j$, there is a forced spatial relation for $(i, j)$ in $\plc{}$.
\end{definition}

Note that in particular a forcing placement is feasible.
Examples of forcing placements are given in
Figure~\ref{intro::figure::placements::placement2} and
Figure~\ref{lowerbound:forcing_placement:example:figure}.

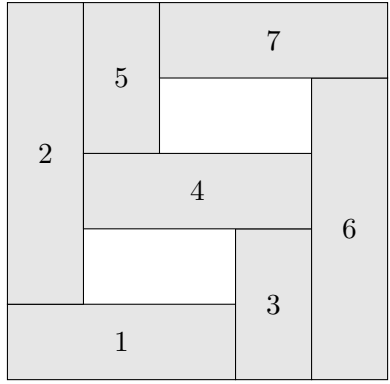
\begin{figure}
\centering
 \begin{tikzpicture}
 \draw (1,1) [fill=gray!20] rectangle (4,2) node[pos=.5] {1};
 \draw (1,2) [fill=gray!20] rectangle (2,6) node[pos=.5] {2};
 \draw (4,1) [fill=gray!20] rectangle (5,3) node[pos=.5] {3};
 \draw (2,3) [fill=gray!20] rectangle (5,4) node[pos=.5] {4};
 \draw (2,4) [fill=gray!20] rectangle (3,6) node[pos=.5] {5};
 \draw (5,1) [fill=gray!20] rectangle (6,5) node[pos=.5] {6};
 \draw (3,5) [fill=gray!20] rectangle (6,6) node[pos=.5] {7};
 \end{tikzpicture}
 \caption{A forcing placement.}
\label{lowerbound:forcing_placement:example:figure}
\end{figure}

\begin{definition}
 Let $\plc{}$ be a forcing placement.
 The \stress{canonical representation} $r_{\plc{}}$ of $\plc{}$ is given by assigning each pair $(i,j)$ to its forced spatial relation.
\end{definition}

Note that by Lemma~\ref{lowerbound:forced_relation:at_most_one:lemma}, the canonical representation is well defined.

\begin{lemma}
\label{lowerbound:canonical_representation:general_placement:lemma}
  Let $n \in \nat$ and let $\plc{}$ be a feasible placement of size $n$.
  Let $r$ be a representation of $\plc{}$ and let $\plc{}^\prime$ be a forcing placement of size $n$ that is represented by $r$.
  Then $\plc{}$ is represented by $r_{\plc{}^\prime}$.
\end{lemma}

\begin{proof}
 Let $i,j \in [n]$ with $i\neq j$.
 There are indices $i = a_1, \ldots, a_k = j$ such that $r_{\plc{}^\prime}(i,j)$ is the only relation of $(a_m, a_{m+1})$ in $\plc{}^\prime$ for all $1 \leq m < k$.
 As $r$ represents $\plc{}^\prime$, we have $r(a_m, a_{m+1}) = r_{\plc{}^\prime}(i,j)$.
 Moreover, $r$ represents $\plc{}$, so $r_{\plc{}^\prime}(i,j)$ is a spatial relation of $(a_m, a_{m+1})$ in $\plc{}$.
 By transitivity, $r_{\plc{}^\prime}(i,j)$ is a spatial relation of $(i,j)$ in $\plc{}$.
 It follows that $r_{\plc{}^\prime}$ represents $\plc{}$.
\end{proof}

\begin{lemma}
 Let $n \in \nat$ and let $C$ be a set of canonical representations of forcing placements of size $n$.
 Furthermore, let $R$ be a complete set of representations for $n$.
 Then there is a complete set of representations $R^\prime$ with $|R^\prime| \leq |R|$ and $C \subseteq R^\prime$.
\end{lemma}

\begin{proof}
Let $R^\prime$ be a complete set of representations with $|R^\prime| \leq |R|$ which minimizes $|C \setminus R^\prime|$.
Note that since $R$ is a candidate for $R^\prime$, $R^\prime$ exists.
We show that $C \setminus R^\prime = \emptyset$.

Otherwise, let $r_\plc{} \in C \setminus R^\prime$ be a canonical representation of a forcing placement $\plc{}$ and let $r^\prime \in R^\prime$ be an element representing $\plc{}$.

If $r^\prime \in C$, let $\plc{}^\prime$ be a forcing placement such that $r^\prime$ is the canonical representation of $\plc{}^\prime$.
Since $r_\plc{} \neq r^\prime$, there are indices $i \neq j$ such that $r_\plc{}(i,j) \neq r^\prime(i,j)$.
Since $\plc{}$ is forcing, there are indices $i=a_1, \ldots, a_k = j$ such that for $1 \leq m < k$, the rectangle pair $(a_m, a_{m+1})$ has only the spatial relation $r_\plc{}(i,j)$ in $\plc{}$.
Hence $r^\prime(a_m, a_{m+1}) = r_\plc{}(i,j)$, as $r^\prime$ represents $\plc{}$.
However, $r^\prime$ is the canonical representation of $\plc{}^\prime$, so for $1 \leq m < k$, the relation $r_\plc{}(i,j)$ is forced for $(a_m, a_{m+1})$ in $\plc{}^\prime$.
By Observation~\ref{forced::transitivity::observation}, the relation $r_\plc{}(i,j)$ is also forced for $i,j$ in $\plc{}^\prime$, contradicting $r_\plc{}(i,j) \neq r^\prime(i,j)$.

This means that $r^\prime \notin C$, and let $R^{\prime\prime} := (R^\prime \setminus \{r^\prime\}) \cup \{r_\plc{}\}$.
By Lemma~\ref{lowerbound:canonical_representation:general_placement:lemma}, every feasible placement that is represented by $r^\prime$ is also represented by $r_\plc{}$.
Hence, $R^{\prime\prime}$ is a complete set of representations with $|R^{\prime\prime}| = |R^{\prime}|$, contradicting the choice of $R^\prime$.
\end{proof}

\begin{corollary}
\label{lowerbound:canoncial_size:corollary}
 Let $n \in \nat$ and let $C$ be a set of canonical representations of forcing placements of size $n$.
 Then any complete set of representations has at least $|C|$ elements.
\end{corollary}

\subsection{Many canonical representations}

Now we get to the main part of the proof:
we show the existence of a large set of canonical representations.
Set $r_{\pi}:=r_{\text{id},\pi}$.

 \begin{lemma}
 \label{lowerbound:forced_pi_condition:lemma}
  Let $\pi$ be a permutation on $[n]$ and let $\plc{}$ be a feasible placement of size $n$.
  Then $\plc{}$ is a forcing placement with $r_\plc{} = r_{\pi}$ iff
  \begin{enumerate}[(i)]
   \item for all $(i,j) \in A(G_{\pi})$, $i$ is only south of $j$ in $\plc{}$, \label{lowerbound:forced_pi_condition:lemma:case1i}
   \item for all $(i,j) \in A(G_{-\pi})$, $i$ is only west of $j$  in $\plc{}$. \label{lowerbound:forced_pi_condition:lemma:case2i}
  \end{enumerate}
 \end{lemma}

 \begin{proof}
  First, we prove that if (\ref{lowerbound:forced_pi_condition:lemma:case1i})
  and (\ref{lowerbound:forced_pi_condition:lemma:case2i}), $\plc{}$ is forcing with $r_\plc{} = r_{\pi}$.
  Let $i,j \in [n]$ with $i<j$.
  By Observation~\ref{lowerbound:minuspi:reachability:observation},
  $j$ is reachable from $i$ in either $G_\pi$ or $G_{-\pi}$, but not both.
  Assume $j$ is reachable from $i$ in $G_{\pi}$.
  Then there is a sequence of vertices $i = a_1, \ldots, a_k = j$ with $(a_m, a_{m+1}) \in A(G_{\pi})$ for $1 \leq m < k$,
  so by (\ref{lowerbound:forced_pi_condition:lemma:case1i}), $i$ south of $j$ is forced.
  Furthermore, since $j$ is reachable from $i$ in $G_{\pi}$, we have $i <_{\pi} j$, so $r_{\pi}(i,j) = \south$.
  The case that $j$ is reachable from $i$ in $G_{-\pi}$ is proven analogously.

  For the other direction, let $\plc{}$ be forcing with $r_\plc{} = r_{\pi}$ and $(i,j) \in A(G_{\pi})$.
  Since $i < j$ and $i <_{\pi} j$, we have $r_\plc{}(i,j) = r_{\pi}(i,j) = \south$, so $i$ is south of $j$.
  It remains to be shown that south is the only spatial relation of $(i,j)$.
  Since $r_\plc{}(i,j) = \south$, $i$ south of $j$ is forced, and there are indices
  $i = a_1, \ldots, a_k = j$ such that $a_m$ is only south of  $a_{m+1}$ in $\plc{}$ for $1 \leq m < k$.
  Since $r_{\pi} = r_\plc{}$ represents $\plc{}$, we have $r_{\pi}(a_m, a_{m+1}) = \south$ for $1 \leq m < k$, so
  $i = a_1 < \cdots < a_k = j$ and $i = a_1 <_{\pi} \cdots <_{\pi} a_k = j$.
  Hence, due to $(i,j) \in A(G_{\pi})$, we have $k=2$, and thus $i$ is only south of $j$.

  Again, the case $(i,j) \in A(G_{-\pi})$ is proven analogously.
 \end{proof}

 Before we prove the main lemma, we need a technical result:

 \begin{lemma}
 \label{lowerbound:j_i_relation_helper:lemma}
  Let $\pi$ be a biplane permutation on $[n]$ with $\wrap{\pi}(n-1) < \wrap{\pi}(n) < n$.
  Let $(j,n) \in A(G_{-\pi})$  such that $j$ has no outgoing edges in $G_{\pi}$ and
  let $i < j$ with $(i,n) \in A(G_{\pi})$.
  Furthermore, let $\plc{}=\placement$ be a forcing placement with $r_\plc{} = r_{\pi}$.
  Then, $i$ is the only index with this property,
  and there is a forcing placement $\plc{}^\prime=\placement[^\prime]$ with
  $r_{\plc{}^\prime} = r_{\pi}$ and $\xmax^\prime(j) < \xmax^\prime(i)$.
 \end{lemma}

 \begin{proof}

    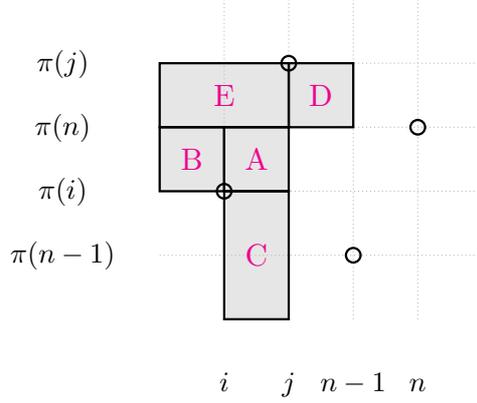
\begin{figure}
      \centering
      \resizebox{0.45\textwidth}{!}{%
      \begin{tikzpicture}[scale=0.8]
         \drawPermutation{i,j,n-1,n}{n-1,i,n,j}
         \fill [\forbiddencolor] (min |- i) rectangle (j);
         \fill [\forbiddencolor] (i |- min) rectangle (j);
         \fill [\forbiddencolor] (j |- n) rectangle (n-1 |- j);
          \begin{pgfonlayer}{fg}
           \draw [thick] (min |- i) rectangle (i |- n) node[pos=.5] {\ref{lowerbound:j_i_relation_helper:lemma:claim:case1}};
           \draw [thick] (i |- min) rectangle (j |- i) node[pos=.5] {\ref{lowerbound:j_i_relation_helper:lemma:claim:case2}};
           \draw [thick] (i) rectangle (j |- n) node[pos=.5] {\ref{lowerbound:j_i_relation_helper:lemma:claim:case3}};
           \draw [thick] (min |- n) rectangle (j) node[pos=.5] {\ref{lowerbound:j_i_relation_helper:lemma:claim:case4}};
           \draw [thick] (j |- n) rectangle (n-1 |- j) node[pos=.5] {\ref{lowerbound:j_i_relation_helper:lemma:claim:case5}};
          \end{pgfonlayer}
      \end{tikzpicture}
      }
      \caption{Configuration with $i < j < n-1 < n$. Gray areas are claimed to be empty.}
      \label{lowerbound:j_i_relation_helper:lemma:figure}
 \end{figure}

First, note that $j \neq n-1$, so $j < n-1$, and due to $i < j < n-1 < n$ and $(n-1,n),(i,n) \in A(G_{\pi})$, we must have
$$n-1 <_{\pi} i <_{\pi} n <_{\pi} j.$$

\noindent
{\bf Claim:}
There is no $l \in [n]$ with either
\begin{enumerate}[\hspace{0.5cm}(A)]
 \item $i < l < j$ and $i <_{\pi} l <_{\pi} n$, or \label{lowerbound:j_i_relation_helper:lemma:claim:case3}
 \item $l < i$ and $i <_{\pi} l <_{\pi} n$, or \label{lowerbound:j_i_relation_helper:lemma:claim:case1}
 \item $i < l < j$ and $l <_{\pi} i$, or \label{lowerbound:j_i_relation_helper:lemma:claim:case2}
 \item $j < l < n-1$ and $n <_{\pi} l <_{\pi} j$, or \label{lowerbound:j_i_relation_helper:lemma:claim:case5}
 \item $l < j$ and $n <_{\pi} l <_{\pi} j$. \label{lowerbound:j_i_relation_helper:lemma:claim:case4}
\end{enumerate}
Figure~\ref{lowerbound:j_i_relation_helper:lemma:figure} illustrates the setting and the five statements.

To prove the Claim,
first observe that an $l$ with (\ref{lowerbound:j_i_relation_helper:lemma:claim:case3}) would contradict $(i,n) \in A(G_{\pi})$.
Next, this implies that an $l$ with (\ref{lowerbound:j_i_relation_helper:lemma:claim:case1})
or with (\ref{lowerbound:j_i_relation_helper:lemma:claim:case2}) would yield (with $i$, $j$ and $n$) the pattern $21\bar 354$, contradicting that $\pi$ is plane.
Third, an $l$ with (\ref{lowerbound:j_i_relation_helper:lemma:claim:case5}) would contradict $(j,n) \in A(G_{-\pi})$.
Finally, this implies that an $l$ with (\ref{lowerbound:j_i_relation_helper:lemma:claim:case4}) would
yield (together with $j$, $n-1$ and $n$) the pattern $45\bar 312$, contradicting that $-\pi$ is plane.
The Claim is proved.

Now, by \eqref{lowerbound:j_i_relation_helper:lemma:claim:case3}, \eqref{lowerbound:j_i_relation_helper:lemma:claim:case1}, and \eqref{lowerbound:j_i_relation_helper:lemma:claim:case2} of the Claim,
there is no $l<j$ with $l\not=i$ and $(l,n)\in A(G_{\pi})$.

Part \eqref{lowerbound:j_i_relation_helper:lemma:claim:case3} and \eqref{lowerbound:j_i_relation_helper:lemma:claim:case4} of the Claim imply that $(i,j) \in A(G_{\pi})$.
Hence, by Lemma \ref{lowerbound:forced_pi_condition:lemma}, $i$ is only south of $j$ in $P$ -- in particular $i$ is not west of $j$ -- so $\xmax(i) > \xmin(j)$.
If $\xmax(j) < \xmax(i)$, there is nothing to show (i.e., set $\plc{}^\prime = \plc{}$), so assume $\xmax(j) \geq \xmax(i)$.

Set $(\xmin^\prime,\ymin^\prime, \xmax^\prime,\ymax^\prime) = (\xmin,\ymin, \xmax,\ymax)$, except for
 $$\xmax^\prime(j) := \frac{\max\{\xmin(i),\xmin(j)\} + \xmax(i)}{2}.$$
 Then
$$\xmax^\prime(j)  < \xmax(i) \leq \xmax(j).$$
 Moreover,
 $$\xmax^\prime(j) \geq \frac{\xmin^\prime(j) + \xmax^\prime(i)}{2} > \xmin^\prime(j).$$
Hence, $\plc{}^\prime$ is still a placement.
Since we only decreased the width of $j$, all only-west and only-east relations of $j$ are still intact.
Moreover, as $j$ has no outgoing edges in $G_{\pi}$,
in order to see that $\plc{}^\prime$ is still forcing with $r_{\plc{}^\prime} = r_{\pi}$,
we only need to verify that for all edges $(k,j) \in A(G_{\pi})$, $k$ is still only south of $j$.
But, by \eqref{lowerbound:j_i_relation_helper:lemma:claim:case3}, \eqref{lowerbound:j_i_relation_helper:lemma:claim:case1}, \eqref{lowerbound:j_i_relation_helper:lemma:claim:case2}
and \eqref{lowerbound:j_i_relation_helper:lemma:claim:case4} of the Claim,
$i$ is the only predecessor of $j$ in $G_{\pi}$,
and since we only reduced $\xmax^\prime(j)$, $j$ is still not east of $i$.
Moreover, we have
$$ \xmax^\prime(j) = \frac{\max\{\xmin^\prime(i),\xmin^\prime(j)\} + \xmax^\prime(i)}{2} \geq \frac{\xmin^\prime(i) + \xmax^\prime(i)}{2} > \xmin^\prime(i).$$
so $j$ is not west of $i$ in $\plc{}^\prime$.
 \end{proof}

\begin{lemma}
\label{lowerbound:placement:lemma}
 Let $\pi$ be a biplane permutation on $[n]$.
 Then there is a forcing placement $\plc{}$ of size $n$ with $r_{\plc{}} = r_{\pi}$.
\end{lemma}

\begin{proof}
 We prove the lemma by induction. The case $n=1$ is trivial, so assume the claim holds for $n \in \nat$ and
 let $\pi$ be a biplane permutation on $[n+1]$.

 First, we consider the case $n <_{\pi} n+1$.
 The other case will later be reduced to this case.
 Let $\pi^\prime$ be the permutation on $[n]$ given by $\pi^\prime(i) := \pi(i)$ if $i <_{\pi} n+1$, and $\pi^\prime(i) := \pi(i) - 1$ otherwise.
 Clearly, for $i,j \in [n]$, we have $i <_{\pi} j \iff i <_{\pi^\prime} j$.
 In particular, $\pi^\prime$ is a biplane permutation, so by the induction hypothesis,
 there is a forcing placement $\plc{}^\prime=\placement[^\prime]$ with $r_{\plc{}^\prime} = r_{\pi^\prime}$.
 Note that $G_{\pi^\prime}$ is an induced subgraph of $G_{\pi}$, and $G_{-\pi^\prime}$ is an induced subgraph of $G_{-\pi}$.
 This means that if we extend $\plc{}^\prime$ to some placement $\plc{}$ of size $n+1$,
 we only need to check edges incident to $n+1$ when applying
 Lemma~\ref{lowerbound:forced_pi_condition:lemma}.

 If $\pi(n+1) = n+1$, then we can just place $n+1$ north of all other rectangles:
 extend $\plc{}^\prime = \placement[^\prime]$ to $\plc{} = \placement$ by
 \begin{align*}
  \xmin(n+1) & := \min_{i \in [n]}\xmin^\prime(i), & \ymin(n+1) & := \max_{i \in [n]}\ymax^\prime(i), \\
  \xmax(n+1) & := \max_{i \in [n]}\xmax^\prime(i), & \ymax(n+1) & := \max_{i \in [n]}\ymax^\prime(i) + 1.
 \end{align*}
 By \stress{extending}, we mean that $\plc{}$ and $\plc{}^\prime$ agree for $i=1,\ldots,n$.
 Then, $n+1$ does not overlap with any rectangle, so $\plc{}$ is a feasible placement.
 For $(i, n+1) \in A(G_{\pi})$, by the construction of $\plc{}$, we have that $i$ is only south of $n+1$ in $\plc{}$.
 Since there are no edges $(i, n+1) \in A(G_{-\pi})$,
 we can apply Lemma~\ref{lowerbound:forced_pi_condition:lemma} to conclude that $\plc{}$ is forcing with $r_\plc{} = r_{\pi}$.

 \begin{figure}
   \centering
   \providecommand{\subfigwidth}{0.47\textwidth}
   \begin{subfigure}[b]{\subfigwidth}
      \centering
      \resizebox{0.95\textwidth}{!}{%
      \begin{tikzpicture}[scale=0.8]
         \drawPermutation{j,n,n+1}{n,n+1,j}
         \fill [\forbiddencolor] (j) rectangle (max);
         \fill [\forbiddencolor] (min |- n+1) rectangle (max |- j);
         \fill [\forbiddencolor] (n |- min) rectangle (max);
      \end{tikzpicture}
      }
      \caption{Situation with $j < n < n+1$ and  \\ $n <_{\pi} n+1 <_{\pi} j$.}
      \label{lowerbound:placement:proof:fig1:case1}
   \end{subfigure}
   \hfill 
   \begin{subfigure}[b]{\subfigwidth}
      \centering
      \resizebox{0.95\textwidth}{!}{%
      \begin{tikzpicture}[scale=0.8]
         \drawPermutation{l,j,n,n+1}{n,n+1,l,j}
         \fill [\forbiddencolor] (j) rectangle (n+1);
      \end{tikzpicture}
      }
      \caption{If there exists a predecessor $l < j$ of $n+1$ in $G_{-\pi}$, then $-\pi$ is not plane.}
      \label{lowerbound:placement:proof:fig1:jonlypred}
   \end{subfigure}
   \caption{Illustrations of orderings of elements in $\pi$. Gray areas do not contain any other elements.}
   \label{lowerbound:placement:proof:fig1}
 \end{figure}
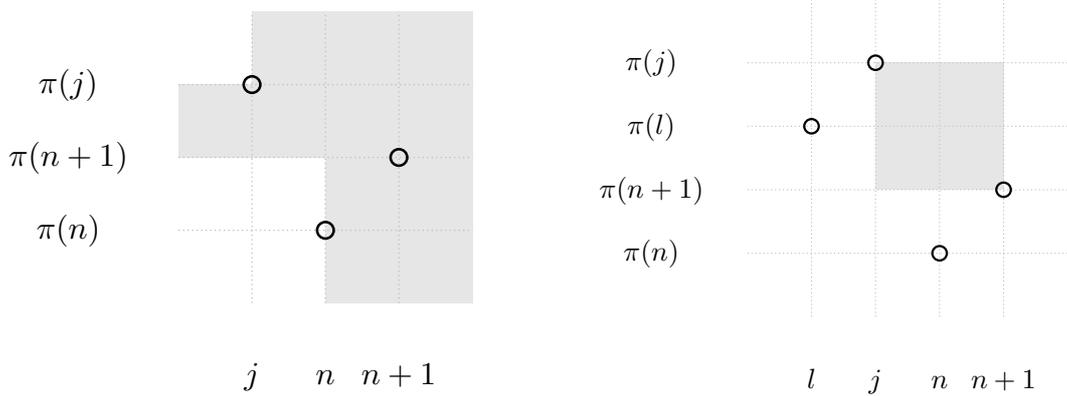

 So assume $\pi(n+1) < n+1$. Let $j$ be maximum with $(j, n+1) \in A(G_{-\pi})$. Note that $j$ exists since
 $n+1$ is reachable from $\pi^{-1}(n+1)$ in $G_{-\pi}$.
 This configuration is illustrated in Figure~\ref{lowerbound:placement:proof:fig1:case1}.
 Then $j$ is the only predecessor of $n+1$ in $G_{-\pi}$: if $l < j$ with $n+1 <_{\pi} l <_{\pi} j$,
 then $(l, j, n, n+1)$ shows that $-\pi$ is not plane, cf.\ Figure~\ref{lowerbound:placement:proof:fig1:jonlypred}.

 Moreover, $j$ has no outgoing edges in $G_{\pi^\prime}$, since if $(j,l) \in A(G_{\pi^\prime})$,
 then $l < n+1$ and $n+1 <_{\pi} l$, so $n+1$ is reachable from $l$ in $G_{-\pi}$, contradicting that
 $j$ is the only predecessor of $n+1$ in $G_{-\pi}$.
 Hence,
 there is no rectangle only north of $j$ in $\plc{}^\prime$,
 and w.l.o.g., we can assume that
 \begin{align} \label{lowerbound:placement:proof:ymaxj}
 \ymax^\prime(j) \geq \max\{1 + \ymax^\prime(i) \;:\; i \in [n] \setminus \{j \}\},
 \end{align}
 since we can increase the height of $j$ as required.
 Increasing the size of rectangles while maintaining a feasible placement does not destroy forced relations,
 so $\plc{}^\prime$ is still forcing with $r_{\plc{}^\prime} = r_{\pi^\prime}$.

 Now, we consider the predecessors of $n+1$ in $G_{\pi}$, which represent the rectangles that $n+1$ has to be north of.
 Let $i$ be minimum with $(i,n+1) \in A(G_{\pi})$.
 Again, $i$ exists since $(n,n+1) \in A(G_{\pi})$.
 If $i < j$, by Lemma~\ref{lowerbound:j_i_relation_helper:lemma}, there is no $i < l < j$ with $(l,n+1) \in A(G_{\pi})$
 and w.l.o.g. we can assume that $\xmax^\prime(j) < \xmax^\prime(i)$.
 Note that (\ref{lowerbound:placement:proof:ymaxj}) can still be assumed.

 We extend $\plc{}^\prime = \placement[^\prime]$ to $\plc{} = \placement$ by
 \begin{align*}
  \xmin(n+1) & := \xmax(j), & \ymin(n+1) & := \ymax(j) - 1, \\
  \xmax(n+1) & := \max_{l \in [n]}\xmax(l), & \ymax(n+1) & := \ymax(j).
 \end{align*}

 First, since $j$ is west of $n$, we have $\xmax(n+1) \ge \xmax(n) > \xmin(n) \ge \xmax(j) = \xmin(n+1)$.
 Furthermore, $n+1$ is east of $j$ and (using \eqref{lowerbound:placement:proof:ymaxj})
 north of all other rectangles, so in particular $n+1$ does not intersect with any rectangle,
 showing that $\plc{}$ is a feasible placement.

 Now, we verify that for all $(l,n+1) \in A(G_{\pi})$ that $l$ is only south of $n+1$,
 and for all $(l,n+1) \in A(G_{-\pi})$ that $l$ is only west of $n+1$.

 Clearly, by construction of $\plc{}$, $j$ is only west of $n+1$ in $\plc{}$,
 and $j$ is the only predecessor of $n+1$ in $G_{-\pi}$.
 As $n+1$ is north of all rectangles other than $j$, it remains to be shown that
 for $(k,n+1) \in A(G_{\pi})$, we have that $k$ is not west of $n+1$ and not east of $n+1$.
 The latter already directly follows from the choice of $\xmax(n+1)$.

 So let $(k,n+1) \in A(G_{\pi})$.
 If $k < j$ we have that $k = i$,
 and by $\xmax(i) > \xmax(j) = \xmin(n+1)$ we have that $i$ is not west of $n+1$.
 Otherwise, i.e., $j < k$, we have $k <_{\pi} n+1 <_{\pi} j$, so $j$ is west of $k$.
 Then $\xmin(n+1) = \xmax(j) \leq \xmin(k) < \xmax(k)$, so $k$ is not west of $n+1$.
 We conclude, using Lemma~\ref{lowerbound:forced_pi_condition:lemma},
 that $\plc{}$ is a forcing placement with $r_{\plc{}} = r_{\pi}$.

Finally, consider the case $n+1 <_{\pi} n$.
Since $\pi$ is biplane, $-\pi$ is biplane as well, and $n <_{-\pi} n + 1$,
so there exists a forcing placement $\plc{}^\prime= \placement[^\prime]$ with $r_{\plc{}^\prime} = r_{-\pi}$.
Now let $\plc{}=(\ymin^\prime, \xmin^\prime, \ymax^\prime, \xmax^\prime)$,
i.e., exchange the role of x-coordinates and y-coordinates in $\plc{}^\prime$.
Since the definition of forcingness is symmetric, clearly $\plc{}$ is still a forcing placement.
Moreover, for $(i,j) \in A(G_{\pi})$, we have $(i,j) \in A(G_{-(-\pi)})$,
so $i$ is only west of $j$ in $\plc{}^\prime$, resulting in $i$ only south of $j$ in $\plc{}$.
Similarly, if $(i,j) \in A(G_{-\pi})$, then $i$ is only south of $j$ in $\plc{}^\prime$, so $i$ is only west of $j$ in $\plc{}$.
Hence, by Lemma~\ref{lowerbound:forced_pi_condition:lemma}, $\plc{}$ is a forcing placement with $r_\plc{} = r_{\pi}$.

\end{proof}

\subsection{Completing the lower bound}

Now, we show that one can apply all permutations on $[n]$ to the canonical representations obtained from Lemma~\ref{lowerbound:placement:lemma},
resulting in a large set of canonical representations.
For permutations $\pi$ and $\rho$ on $[n]$, we denote by $\rho(\pi)$ the permutation on $[n]$ given by
$$\bigl(\rho(\pi)\bigr)(i) = \rho(\pi(i)).$$
{
\providecommand{\biplaneseqpair}{\Pi}

\begin{lemma}
\label{lowerbound:large_canonical_set:lemma}
 Let $n \in \nat$ and
 let
 $$\biplaneseqpair{} = \{ (\pi,\rho) \;:\; \text{$\pi$ and $\rho$ are permutations on $[n]$, $\rho$ is biplane} \}.$$
 Furthermore, let $$C = \{ r_{\pi,\rho(\pi)} \;:\; (\pi,\rho) \in \biplaneseqpair{} \}.$$
 Then, $|C| = |\biplaneseqpair{}|$, and $C$ consists of canonical representations only.
\end{lemma}

\begin{proof}
Clearly, $|C| = |\biplaneseqpair{}|$ holds by definition of $r_{\pi,\rho(\pi)}$.

 We show that $C$ consists of canonical representations only.
 Let $(\pi,\rho) \in \biplaneseqpair{}$.
 Since $\rho$ is biplane, by Lemma~\ref{lowerbound:placement:lemma}, there is a forcing placement $\plc{}$ such that
 $r_{\plc{}} = r_{\rho}$. We now show that permuting the rectangles in $\plc{}$ according to $\pi$ yields a forcing placement
 $\plc{}^\prime$ with $r_{\plc{}^\prime} = r_{\pi,\rho(\pi)}$.
 Let $\plc{} = \placement$ and define $\plc{}^\prime = \placement[^\prime]$ by, for $i \in [n]$,
 \begin{align*}
  \xmin^\prime(i) &:= \xmin(\pi(i)), &
  \ymin^\prime(i) &:= \ymin(\pi(i)), \\
  \xmax^\prime(i) &:= \xmax(\pi(i)), &
  \ymax^\prime(i) &:= \ymax(\pi(i)). \\
 \end{align*}
Obviously, $\plc{}^\prime$ is still a forcing placement.
Furthermore, for $i,j\in [n]$ with $i \neq j$, we have
\begin{align*}
 r_{\plc{}^\prime}(i,j) &= r_{\plc{}}(\pi(i),\pi(j)) \\
                        &= r_{\rho}(\pi(i), \pi(j)) \\
                        &= r_{\text{id}, \rho}(\pi(i), \pi(j)) \\
                        &= r_{\pi, \rho(\pi)}(i,j).
\end{align*}
\end{proof}
} 

\begin{theorem}
 Let $c = 4 + 2 \sqrt2 \geq \smallClb{}$.
 Every complete set of representations for $n$ has
 $\Omega \bigl( n! \cdot \frac{c^n}{n^4} \bigr)$ elements.
\end{theorem}

\begin{proof}
By Lemma~\ref{lowerbound:large_canonical_set:lemma}, there is a set $C$ of canonical representations of size $n$ that contains a separate element for each
pair $\pi, \rho$ of permutations where $\rho$ is biplane.
By Theorem \ref{thm_numberofbiplaneperm}, the number of biplane permutations is $\Theta(\frac{(4 + 2 \sqrt2)^n}{n^4})$.
The result now follows from Corollary~\ref{lowerbound:canoncial_size:corollary}.
\end{proof}

\begin{figure}
\centering
 \begin{tikzpicture}
 \draw (1,1) [fill=gray!20] rectangle (3,2) node[pos=.5] {1};
 \draw (1,2) [fill=gray!20] rectangle (2,4) node[pos=.5] {2};
 \draw (3,1) [fill=gray!20] rectangle (5,3) node[pos=.5] {3};
 \draw (2,3) [fill=gray!20] rectangle (4,4) node[pos=.5] {4};
 \draw (4,3) [fill=gray!20] rectangle (5,4) node[pos=.5] {5};
 \end{tikzpicture}
 \caption{A feasible placement that is not representable by a canonical representation.}
 \label{lowerbound:not_tight:figure}
\end{figure}
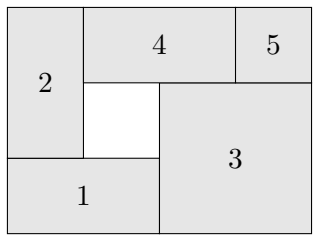

\bigskip Finally, we observe that the construction is not tight:

\begin{lemma}
\label{lowerbound:not_tight:lemma}
 Let $n \in \nat$ with $n \geq 5$. Then the cardinality of any complete set of representations for $n$ is strictly larger than $n!$ times the number of biplane permutations on $[n]$.
\end{lemma}
\begin{proof}
 First, we prove the case $n=5$.
 Consider the feasible placement $\plc{}$ as depicted in Figure~\ref{lowerbound:not_tight:figure}.
 We show that $\plc{}$ is not representable by a canonical representation.

 So suppose that $\plc{}^\prime=\placement[^\prime]$ is a forcing placement such that $\plc{}$ is represented by $r_{\plc{}^\prime}$.
The pair $(1,5)$ is the only pair without a forced relation in $\plc{}$.
Moreover, there is no $1 < i < 5$ such that $(1,i)$ and $(i,5)$ have the same relation in $\plc{}$.
Hence, the only way to force a relation for $(1,5)$ in $\plc{}^\prime$ is to either let 1 be only west of 5 or let 1 be only south of 5 in $\plc{}^\prime$.

For all pairs $1 \leq i < j \leq 4$, there is no $k$ such that $(i,k)$ and $(k,j)$ have the same relation in $\plc{}$.
Hence, all such $(i,j)$  may only have one relation in $\plc{}^\prime$ as well.
Since 3 is south of 4, but not east of 4, we have $\xmin^\prime(3) < \xmax^\prime(4)$.
This implies
$$\xmax^\prime(1) \leq \xmin^\prime(3) < \xmax^\prime(4) \leq \xmin^\prime(5),$$
so 1 is west of 5 in $\plc{}^\prime$.
Similarly, 2 is west of 3, but not north of 3, so we have $\ymin^\prime(2) < \ymax^\prime(3)$.
Then
$$\ymax^\prime(1) \leq \ymin^\prime(2) < \ymax^\prime(3) \leq \ymin^\prime(5),$$
so 1 is south of 5 in $\plc{}^\prime$.
This contradicts that $1$ and $5$ have only one relation in $\plc{}^\prime$.

For the case $n > 5$, the same argument works after adding $n - 5$ rectangles to $\plc{}$ that are east of $\{1, \ldots, 5\}$.
\end{proof}

\section{Acknowledgements}
The authors are thankful to the On-Line Encyclopedia of Integer Sequences \cite{oeisplane}, which drew their attention to plane permutations.

\bibliographystyle{plain}
\bibliography{rectpack}   

\end{document}